\documentclass[12pt,leqno]{amsart}
\usepackage{amsmath,amssymb}
\usepackage{calrsfs}
\usepackage{url}
\newcommand{\F}{{\mathbb{F}}}
\newcommand{\Z}{{\mathbb{Z}}}

\newcommand{\N}{{\mathbb{N}}}

\newcommand{\C}{{\mathbb{C}}}
\newcommand{\ba}{{\mathbf{a}}}
\newcommand{\bG}{{\mathbf{G}}}
\newcommand{\bB}{{\mathbf{B}}}
\newcommand{\bL}{{\mathbf{L}}}
\newcommand{\bT}{{\mathbf{T}}}
\newcommand{\bU}{{\mathbf{U}}}

\newcommand{\cH}{{\mathcal{H}}}
\newcommand{\cP}{{\mathcal{P}}}
\newcommand{\cO}{{\mathcal{O}}}
\newcommand{\cS}{{\mathcal{S}}}
\newcommand{\cU}{{\mathcal{U}}}
\newcommand{\cPR}{{\mathcal{PR}}}
\newcommand{\fF}{{\mathfrak{F}}}
\newcommand{\fS}{{\mathfrak{S}}}
\newcommand{\fe}{{\mathfrak{e}}}
\newcommand{\fp}{{\mathfrak{p}}}
\newcommand{\fb}{{\underline{\mathfrak{b}}}}
\newcommand{\fu}{{\underline{\mathfrak{u}}}}

\newcommand\Rst{{^*\!R}}
\newcommand{\Fix}{{\operatorname{Fix}}}
\newcommand{\Irr}{{\operatorname{Irr}}}
\newcommand{\Hom}{{\operatorname{Hom}}}
\newcommand{\Ind}{{\operatorname{Ind}}}
\newcommand{\End}{{\operatorname{End}}}
\newcommand{\rad}{{\operatorname{rad}}}
\newcommand{\St}{{\operatorname{St}}}
\newcommand{\GU}{{\operatorname{GU}}}
\newcommand{\GL}{{\operatorname{GL}}}
\newcommand{\charak}{{\operatorname{char}}}
\renewcommand{\leq}{\leqslant}
\renewcommand{\geq}{\geqslant}

\newtheorem{thm}{Theorem}[section]
\newtheorem{lem}[thm]{Lemma}
\newtheorem{cor}[thm]{Corollary}
\newtheorem{prop}[thm]{Proposition}

\theoremstyle{definition}
\newtheorem{exmp}[thm]{Example}

\newtheorem{abs}[thm]{}
\theoremstyle{remark}
\newtheorem{rem}[thm]{Remark}

\begin{document}
% ------------ for elsart --------------
%\begin{frontmatter}
%\title{On the modular composition factors of the Steinberg representation}
%
%\author{Meinolf Geck}
%\address{IAZ - Lehrstuhl f\"ur Algebra, Universit\"at Stuttgart, 
%Paffenwaldring 57, 70569 Stuttgart, Germany}
%\ead{meinolf.geck@mathematik.uni-stuttgart.de}
%
%\begin{abstract} Let $G$ be a finite group of Lie type and $\St_k$ be 
%the Steinberg representation of $G$, defined over a field $k$. We are 
%interested in the case where $k$ has prime characteristic~$\ell$ 
%and $\St_k$ is reducible. Tinberg has shown that the socle of $\St_k$ is 
%always simple.  We give a new proof of this result in terms of the Hecke
%algebra of $G$ with respect to a Borel subgroup and show how to identify 
%the simple socle of $\St_k$ among the principal series representations 
%of~$G$. Furthermore, we determine the composition length of $\St_k$ when 
%$G=\GL_n(q)$ or $G$ is a finite classical group and $\ell$ is a so-called 
%linear prime.
%\end{abstract}
%
%\begin{keyword}
%Finite groups of Lie type, Steinberg representation, Hecke algebra, modular 
%representations.
%\MSC Primary 20C33; Secondary 20C20
%\end{keyword}
%\end{frontmatter}

\makeatletter
\renewenvironment{proof}[1][Proof]{\par
  \normalfont \topsep6\p@\@plus6\p@\relax \trivlist
  \item[\hskip\labelsep \itshape #1\@addpunct{.}]
\ignorespaces}{\endtrivlist\@endpefalse}
\makeatother

\title{On the $\ell$-modular composition factors of the Steinberg 
representation}

\author{Meinolf Geck}
\address{IAZ---Lehrstuhl f\"ur Algebra, Universit\"at Stuttgart, 
Paffenwaldring 57, 70569 Stutt\-gart, Germany}
\email{meinolf.geck@mathematik.uni-stuttgart.de}

\begin{abstract} Let $G$ be a finite group of Lie type and $\St_k$ be 
the Steinberg representation of $G$, defined over a field $k$. We are 
interested in the case where $k$ has prime characteristic~$\ell$ 
and $\St_k$ is reducible. Tinberg has shown that the socle of $\St_k$ is 
always simple.  We give a new proof of this result in terms of the Hecke
algebra of $G$ with respect to a Borel subgroup and show how to identify 
the simple socle of $\St_k$ among the principal series representations 
of~$G$. Furthermore, we determine the composition length of $\St_k$ when 
$G=\GL_n(q)$ or $G$ is a finite classical group and $\ell$ is a so-called 
linear prime.
\end{abstract}
\maketitle

%\centerline{\it To the memory of Sandy Green}
%%%%%%%%%%%%%%%%%%%%%%%%%%%%%%%%%%%%%%%%%%%%%%%%%%%%%%%%%%%%%%%%%%%%%%%
\section{Introduction} \label{sec0}
Let $G$ be a finite group of Lie type and $\St_k$ be the Steinberg 
representation of $G$, defined over a field $k$. Steinberg 
\cite{St57} showed that $\St_k$ is irreducible if and only if $[G:B]1_k
\neq 0$ where $B$ is a Borel subgroup of $G$. We shall be concerned here 
with the case where $\St_k$ is reducible. There is only very little general 
knowledge about the structure of $\St_k$ in this case. We mention the 
works of Tinberg \cite{tin1} (on the socle of $\St_k)$, Hiss \cite{hiss0} 
and Khammash \cite{kham} (on trivial composition factors of $\St_k$) and 
Gow \cite{gow} (on the Jantzen filtration of $\St_k$). 

One of the most 
important open questions in this respect seems to be to find a suitable 
bound on the length of a composition series of $\St_k$. Typically, this 
problem is related to quite subtle information about decomposition 
numbers; see, for example, Landrock--Michler \cite{LaMi} and 
Okuyama--Waki \cite{ow2} where this is solved for groups with a $BN$-pair 
of rank~$1$. For groups of larger $BN$-rank, this problem is completely open.

In this paper, we discuss two aspects of this problem. 

Firstly, Tinberg \cite{tin1} has shown that the socle of $\St_k$ is always 
simple, using results of Green \cite{Green2} applied to the endomorphism 
algebra of the permutation module $k[G/U]$ where $U$ is a maximal 
unipotent subgroup. After some preparations in Sections~\ref{sec2},
we show in Section~\ref{sec2a} that a similar argument works with $U$
replaced by $B$. Since the corresponding 
endomorphism algebra (or ``Hecke algebra'') is much easier to describe 
and its representation theory is quite well understood, this provides 
new additional information. For example, if $G=\GL_n(q)$, then we can 
identify the partition of $n$ which labels the socle of $\St_k$ in 
James' \cite{Jam1} parametrisation of the unipotent simple modules of $G$;
see Example~\ref{socgln}. Quite remarkably, this involves a particular
case of the ``Mullineux involution''~---~and an analogue of this involution
for other types of groups.

In another direction, we consider the partition of the simple $kG$-modules 
into Harish-Chandra series, as defined by Hiss \cite{hiss2}. Dipper and 
Gruber \cite{digr} have developed a quite general framework for this purpose,
in terms of so-called ``projective restriction systems''. In 
Section~\ref{sec2b}, we shall present a simplified, self-contained version 
of parts of this framework which is tailored towards applications to $\St_k$.
This yields, first of all, new proofs of some of the results of Szechtman 
\cite{Sz1} on $\St_k$ for $G=\GL_n(q)$; moreover, in Example~\ref{stgln}, 
we obtain an explicit formula for the composition length of $\St_k$ in this 
case. Analogous results are derived for groups of classical type in the 
so-called ``linear prime'' case, based on \cite{lymgh}, \cite{grub}, 
\cite{grhi}. For example, $\St_k$ is seen to be multiplicity-free with a 
unique simple quotient in these cases~---~a property which does not hold in 
general for non-linear primes.

%%%%%%%%%%%%%%%%%%%%%%%%%%%%%%%%%%%%%%%%%%%%%%%%%%%%%%%%%%%%%%%%%%%%%%%
\section{The Steinberg module and the Hecke algebra} \label{sec2}

Let $G$ be a finite group and $B,N\subseteq G$ be subgroups which satisfy 
the axioms for an ``algebraic group with a split $BN$-pair'' in 
\cite[\S 2.5]{Ca2}. We just recall explicitly those properties of $G$ which 
will be important for us in the sequel. Firstly, there is a prime number
$p$ such that we have a semidirect product decomposition $B=U\rtimes H$ 
where $H=B\cap N$ is an abelian group of order prime to $p$ and $U$ is a 
normal $p$-subgroup of $B$. The group $H$ is normal in $N$ and $W=N/H$ is 
a finite Coxeter group with a canonically defined generating set $S$; let 
$l\colon W\rightarrow \N_0$ be the corresponding length function. For each 
$w\in W$, let $n_w\in N$ be such that $Hn_w=w$. Then we have the Bruhat 
decomposition
\[G=\coprod_{w\in W} Bn_wB=\coprod_{w\in W} Bn_wU,\]
where the second equality holds since $B=U\rtimes H$ and $H$ is normal in 
$N$.

Next, there is a refinement of the above decomposition. Let $w_0\in W$ be 
the unique element of maximal length; we have $w_0^2=1$. Let $n_0\in N$ be 
a representative of $w_0$ and $V:=n_0^{-1}Un_0$; then $U\cap V=H$. For 
$w\in W$, let $U_w:=U\cap n_w^{-1}Vn_w$. (Note that $V$, $U_w$ do not 
depend on the choice of $n_0$, $n_w$ since $U$ is normalised by $H$.) Then 
we have the following sharp form of the Bruhat decomposition:
\[G=\coprod_{w\in W} Bn_wU_w,\qquad \mbox{with uniqueness of expressions},\]
that is, every $g\in Bn_wB$ can be uniquely written as $g=bn_wu$ where $
b\in B$ and $u\in U_w$. It will occasionally be useful to have a version
of this where the order of factors is reversed: By inverting elements, we 
obtain 
\[G=\coprod_{w\in W} U_{w^{-1}}n_wB,\qquad \mbox{with uniqueness of 
expressions}.\]

Now let $R$ be a commutative ring (with identity $1_R$) and $RG$ be the 
group algebra of $G$ over $R$. All our $RG$-modules will be left modules
and, usually, finitely generated. For any subgroup $X\subseteq G$, we 
denote by $R_X$ the trivial $RX$-module. Let $\fb:=\sum_{b\in B} b\in RG$. 
Then $RG\fb$ is an $RG$-module which is canonically isomorphic to the 
induced module $\Ind_B^G(R_B)$. In fact, this realization of $\Ind_B^G(R_B)$
will be particularly suited for our purposes, as we shall see below 
when we consider its endomorphism algebra.

\begin{thm}[Steinberg \protect{\cite{St57}}] \label{stein} 
Consider the $RG$-submodule 
\[\St_R:=RG\fe\subseteq RG\fb\qquad \mbox{where}\qquad \fe:=
\sum_{w\in W} (-1)^{l(w)}n_w\fb.\]
\begin{itemize}
\item[{\rm (i)}] The set $\{u\fe\mid u\in U\}$ is an $R$-basis of $\St_k$. 
Thus, $\St_R$ is free over $R$ of rank $|U|$.
\item[{\rm (ii)}] Assume that $R$ is a field. Then $\St_R$ is an (absolutely) 
irreducible $RG$-module  if and only if $[G:B]1_R\neq 0$.
\end{itemize}
\end{thm}

({\em Note about the proof:} Steinberg uses a list of $14$ axioms concerning
finite Chevalley groups and their twisted versions; all these axioms are 
known to hold in the abstract setting of ``algebraic groups with a split 
$BN$-pair''; see \cite[\S 2.5 and Prop.~2.6.1]{Ca2}.)

\medskip
When $R=k$ is a field, Tinberg \cite[Theorem~4.10]{tin1} determined the 
socle of $\St_k$ and showed that this is simple. An essential ingredient 
in Tinberg's proof are Green's results \cite{Green2} on the Hom functor, 
applied to the endomorphism algebra of the $kG$-module $kG\fu_1$, where 
$\fu_1=\sum_{u\in U} u$. There is a description of this algebra in terms of 
generators and relations, and this is used in order to study 
the indecomposable direct summands of $kG\fu_1$. Our aim is to show 
that an analogous argument works directly with the module $kG\fb$, whose 
endomorphism algebra has a much simpler description. 

So let again $R$ be any commutative ring (with $1_R$), and 
consider the Hecke algebra
\[ \cH_R=\cH_R(G,B):=\End_{RG}(RG\fb)^{\text{opp}}.\]
Following Green \cite{Green2}, a connection between (left) $RG$-modules 
and (left) $\cH_R$-modules is established through the Hom functor 
\[ \fF_R\colon RG\mbox{-modules} \rightarrow \cH_R\mbox{-modules}, 
\qquad M\mapsto \fF_R(M):=\Hom_{RG}(RG\fb,M),\]
where $\fF_R(M)$ is a left $\cH_R$-module via $\cH_R \times \fF_R(M) 
\rightarrow \fF_R(M)$, $(h,\alpha)\mapsto \alpha \circ h$. (See also 
\cite[\S 2.C]{my98} where this Hom functor is studied in a somewhat 
more general context.) Note that, by \cite[(1.3)]{Green2}, we have an 
isomorphism of $R$-modules 
\[\Fix_B(M):=\{m \in M\mid b.m=m \mbox{ for all $b\in B$}\}
\stackrel{\sim}{\longrightarrow} \fF_R(M),\]
which takes $m\in \Fix_B(M)$ to the map $\theta_m\colon 
RG\fb \rightarrow M$, $g\fb\mapsto gm$ ($g\in G$). 

Now, $\cH_R$ is free over $R$ with a standard basis $\{T_w\mid w \in W\}$, 
where the endomorphism $T_w\colon RG\fb\rightarrow RG\fb$ is given by 
\[ T_w(g\fb)=\sum_{g'B\in G/B \text{ with } g^{-1}g'\in Bn_wB} g'\fb
\qquad\qquad(g\in G).\] 
The multiplication is given as follows. Let $w\in W$, $s\in S$ and write 
$q_s:=|U_s|1_R$. Then 
\[T_sT_w=\left\{\begin{array}{cl} T_{sw} & \qquad \mbox{if $l(sw)>l(w)$},\\
q_sT_{sw}+(q_s-1)T_w & \qquad \mbox{if $l(sw)<l(w)$}. \end{array}\right.\]
(See \cite[\S 8.4]{gepf} for a proof and further details.) The crucial step
in our discussion consists of determining the $\cH_R$-module $\fF_R(\St_R)$. 
This will rely on the following basic identity, an analogous version of which 
was shown by Tinberg \cite[4.9]{tin1} for the action of the standard basis
elements of the endomorphism algebra of $kG\fu_1$ (where $k$ is a field). 

\begin{lem} \label{othm2b} We have $T_w(\fe)=(-1)^{l(w)}\fe$ for 
all $w\in W$.
\end{lem}

\begin{proof} It is sufficient to show that $T_s(\fe)=-\fe$ for $s\in S$.
Now, by definition, we have
\[T_s(\fe)=\sum_{w\in W} (-1)^{l(w)}T_s(n_w\fb)=\sum_{w\in W} (-1)^{l(w)}
\sum_{gB} g\fb\]
where the second sum runs over all cosets $gB\in G/B$ such that $n_w^{-1}g
\in Bn_sB$. By the sharp form of the Bruhat decomposition, a set of
representatives for these cosets is given by $\{n_{ws}\}\cup 
\{n_wvn_s\mid 1\neq v\in U_s\}$. This yields 
\[T_s(\fe)=\sum_{w\in W} (-1)^{l(w)} n_{ws}\fb+\sum_{w\in W} 
\sum_{1\neq v \in U_s} (-1)^{l(w)}n_wvn_s\fb.\]
Since $l(ws)=l(w)+1$ for $w \in W$, the first sum equals $-\fe$. 
So it suffices to show that 
\[\sum_{w\in W} (-1)^{l(w)} \kappa_w=0 \qquad \mbox{where}
\qquad\kappa_w:=\sum_{1\neq v\in U_s} n_wvn_s \fb.\]
Let $1\neq v\in U_s$. Since $P_s=B\cup Bn_sB$ is a parabolic 
subgroup of $G$, we have $n_s^{-1}vn_s\in P_s$. By the sharp form of 
the Bruhat decomposition, $n_s^{-1}vn_s\not\in B$ and so $n_s^{-1}vn_s=
v'n_sb_v$ where $v'\in U_s$ and $b_v\in B$ are uniquely determined by~$v$. 
Hence, we have $n_wvn_s\fb=n_wn_sv'n_sb_v\fb=n_{ws}v'n_s\fb$ and so 
\[\kappa_w=\sum_{1\neq v\in U_s} n_wvn_s \fb=\sum_{1\neq v\in U_s} n_{ws}
v'n_s\fb=\sum_{1\neq v\in U_s} n_{ws}vn_s\fb=\kappa_{ws}, \]
where the third equality holds since, by \cite[2.1]{tin1}, the 
map $v\mapsto v'$ is a permutation of $U_s\setminus\{1\}$. 
Consequently, we have 
\[\sum_{w\in W}(-1)^{l(w)}\kappa_w=\sum_{w\in W}(-1)^{l(w)}\kappa_{ws}=
\sum_{w\in W} (-1)^{l(ws)}\kappa_w=-\sum_{w\in W} (-1)^{l(w)}\kappa_w.\]
We conclude that the identity $\sum_{w\in W}(-1)^{l(w)} \kappa_w=0$ holds
if $R=\Z$. For $R$ arbitrary, we apply the canonical map $\Z G\rightarrow 
RG$ and conclude that this identity remains valid in $RG$. (Such an 
argument was already used by Steinberg in the proof of 
\cite[Lemma~2]{St57}.) \qed
\end{proof}

\begin{cor} \label{othm2c} We have $\fF_R(\St_R)=\langle \theta_{\fu_1\fe}
\rangle_R$ and the action of $\cH_R$ on this $R$-module of rank~$1$ is 
given by the algebra homomorphism $\varepsilon\colon \cH_R\rightarrow R$, 
$T_w\mapsto (-1)^{l(w)}$. 
\end{cor}

\begin{proof} Since $\{u\fe\mid u \in U\}$ is an $R$-basis of $\St_R$ and 
$H$ normalises $U$, we have $\Fix_B(\St_R)=\langle \fu_1\fe\rangle_R$ 
and so $\fF_R(\St_R)=\langle \theta_{\fu_1\fe} \rangle_R$. It remains to 
show that $T_s.\theta_{\fu_1\fe}=-\theta_{\fu_1\fe}$ for all $s\in S$. 
Since $\fF_R(\St_R)$ has rank~$1$, we have $T_s.\theta_{\fu_1\fe}=\lambda 
\theta_{\fu_1\fe}$ for some $\lambda \in R$. This implies that
\[\lambda \fu_1\fe=\lambda\theta_{\fu_1\fe}(\fb)=(T_s.\theta_{\fu_1\fe})
(\fb)=(\theta_{\fu_1\fe}\circ T_s)(\fb)=\sum_{gB\in G/B \text{ with } 
g\in Bn_sB} g\fu_1\fe.\]
Thus, the assertion that $\lambda=-1$ is equivalent to the following 
identity:
\begin{equation*}
\sum_{gB\in G/B \text{ with } g\in Bn_sB} g\fu_1\fe=-\fu_1\fe.\tag{$*$}
\end{equation*}
One can either work this out directly by an explicit computation (using 
the various ``structural equations'' in \cite{St57}, \cite{tin1}), or 
one can argue as follows. Lemma~\ref{othm2b} shows that 
\[\lambda\theta_{\fu_1\fe}(\fe)=(T_s.\theta_{\fu_1\fe})(\fe)=
(\theta_{\fu_1\fe}\circ T_s)(\fe)=- \theta_{\fu_1\fe}(\fe).\]
Furthermore, by Steinberg \cite[Lemma~2]{St57}, we have 
\[\theta_{\fu_1\fe}(\fe)=\sum_{w\in W} (-1)^{l(w)}n_w\fu_1\fe=\sum_{w\in W}
\sum_{u\in U} (-1)^{l(w)}n_wu\fe=[G:B]\fe.\]
Thus, if $R=\Z$, then $\theta_{\fu_1e}(\fe)\neq 0$; consequently, in this 
case, we do have $\lambda=-1$ and so ($*$) holds for $R=\Z$. As in the above 
proof, it follows that ($*$) holds for any $R$. \qed
\end{proof}

\begin{rem} \label{othm2a} Assume that $R$ is an integral domain and that 
we have a decomposition $RG\fb=M_1\oplus \cdots \oplus M_r$ where each 
$M_j$ is an indecomposable $RG$-module. Since $\{T_w\mid w\in W\}$ is an
$R$-basis of $\cH_R$, Lemma~\ref{othm2b} implies that every idempotent in
$\cH_R$ either acts as the identity on $\St_R$ or as~$0$. It easily 
follows that there is a unique $i$ such that $\St_R \subseteq M_i$. In 
analogy to Tinberg \cite[4.10]{tin1}, we call this $M_i$ the {\em Steinberg 
component} of $RG\fb$. 

As observed by Khammash \cite[(3.10)]{kham0}, the above argument actually
shows that 
\[ \St_R\subseteq \{m\in RG\fb\mid T_w(m)=(-1)^{l(w)}m \mbox{ for all
$w\in W$}\}\subseteq M_i.\]
Then Khammash \cite[Cor.~\S 3]{kham} proved that the first inequality 
always is an equality. 
\end{rem}

\begin{rem} \label{truegf} At some places in the discussion below, it
will be convenient or even necessary to assume that $G$ is a true finite 
group of Lie type, as in \cite[\S 1.18]{Ca2}. Thus, using the notation in
[{\em loc.\ cit.}], we have $G=\bG^F$ where $\bG$ is a connected reductive 
algebraic group $\bG$ over $\overline{\F}_p$ and 
$F\colon \bG \rightarrow \bG$ is an endomorphism such that some power of $F$ 
is a Frobenius map. Then the ingredients of the $BN$-pair in $G$ will also 
be derived from $\bG$: we have $B=\bB^F$ where $\bB$ is an $F$-stable Borel 
subgroup of $\bG$ and $H=\bT_0^F$ where $\bT_0$ is an $F$-stable maximal 
torus contained in $\bB$; furthermore, $N=N_{\bG}(\bT_0)^F$ and $U=\bU^F$ 
where $\bU$ is the unipotent radical of $\bB$. This set-up leads to the 
following two definitions.

(1) We define a real number $q>0$ by the condition that $|U|=q^{|\Phi|/2}$ 
where $\Phi$ is the root system of $\bG$ with respect to $\bT_0$. Then there 
are positive integers $c_s>0$ such that $|U_s|=q^{c_s}$ for all 
$s\in S$; see \cite[\S 2.9]{Ca2}. Consequently, the relations in $\cH_R$ 
read:
\[T_sT_w=\left\{\begin{array}{cl} T_{sw} & \qquad \mbox{if $l(sw)>l(w)$},\\
q^{c_s}T_{sw}+(q^{c_s}-1)T_w & \qquad \mbox{if $l(sw)<l(w)$}. 
\end{array}\right.\]

(2) The commutator subgroup $[\bU,\bU]$ is an $F$-stable closed
connected normal subgroup of $\bU$. We define the subgroup $U^*:=[\bU,\bU]^F
\subseteq U$. Then $[U,U]\subseteq U^*$. (In most cases, we have $U^*=[U,U]$
but there are exceptions when $q$ is very small; see the remarks in
\cite[p.~258]{lecst}.) The definition of $U^*$ will be needed in 
Section~\ref{sec2b}, where we shall consider group homomorphisms 
$\sigma\colon U \rightarrow R^\times$ such that $U^* \subseteq \ker(\sigma)$.
\end{rem}

%%%%%%%%%%%%%%%%%%%%%%%%%%%%%%%%%%%%%%%%%%%%%%%%%%%%%%%%%%%%%%%%%%%%%%%
\section{The socle of the Steinberg module} \label{sec2a}

We keep the general setting of the previous section and assume now that 
$R=k$ is a field and $\ell:=\charak(k)\neq p$; thus, the parameters of 
the endomorphism algebra $\cH_k$ satisfy $q_s\neq 0$ for all $s\in S$. With 
this assumption, we have the following two results:
\begin{itemize}
\item[{\rm (A)}] Every simple submodule of $kG\fb$ is isomorphic to a 
factor module of $kG\fb$, and vice versa; see Hiss 
\cite[Theorem~5.8]{hiss2} where this is proved much more generally.
\item[{\rm (B)}] $\cH_k$ is a quasi-Frobenius algebra. Indeed,  since
$q_s\neq 0$ for all $s\in S$, $\cH_k$ even is a symmetric algebra with
respect to the trace form $\tau\colon \cH_k\rightarrow k$ defined by 
$\tau(T_1)=1$ and $\tau(T_w)=0$ for $w\neq 1$; see, e.g., 
\cite[8.1.1]{gepf}. 
\end{itemize}
It was first observed in \cite[\S 2]{lymgh} that, in this situation, the 
results of Green \cite{Green2} apply (the original applications of which 
have been to representations of $G$ over fields of characteristic equal 
to~$p$). Let us denote by $\Irr_k(G)$ the set of all simple $kG$-modules 
(up to isomorphism) and by $\Irr_k(G\mid B)$ the set of all $Y \in
\Irr_k(G)$ such that $Y$ is isomorphic to a submodule of $kG\fb$. In the 
general framework of \cite{hiss2}, this is the Harish-Chandra series 
consisting of the {\em unipotent principal series representations} of~$G$. 
Furthermore, let $\Irr(\cH_k)$ be the set of all simple $\cH_k$-modules (up 
to isomorphism). Then, by \cite[Theorem~2]{Green2}, the Hom functor 
$\fF_k$ induces a bijection
\begin{equation*}
\Irr_k(G\mid B) \stackrel{\sim}{\longrightarrow} \Irr(\cH_k), \qquad 
M\mapsto \fF_k(M)=\Hom_{kG}(kG\fb,M); \tag{$\spadesuit$}\]
furthermore, by \cite[Theorem~1]{Green2}, each indecomposable direct 
summand of $kG\fb$ has a simple socle and a unique simple quotient. 
Combined with Remark~\ref{othm2a}, this already shows that $\St_k$ 
has a simple socle. More precisely, we have:

\begin{thm}[Cf.\ Tinberg \protect{\cite[4.10]{tin1}}] \label{othm2} Let 
$Y \subseteq \St_k$ be a simple submodule. Then $\fu_1\fe\in Y$ and, 
hence, $Y$ is uniquely determined. Furthermore, $\dim \fF_k(Y)=1$ and the 
action of $\cH_k$ on $\fF_k(Y)$ is given by the algebra homomorphism 
$\varepsilon\colon \cH_k\rightarrow k$, $T_w\mapsto (-1)^{l(w)}$. 
\end{thm}

\begin{proof} By composing any map in $\fF_k(Y)$ with the inclusion $Y
\subseteq \St_k$, we obtain an embedding $\fF_k(Y)\hookrightarrow 
\fF_k(\St_k)$ and we identify $\fF_k(Y)$ with a subset of $\fF_k(\St_k)$ in 
this way. Now $Y\subseteq \St_k\subseteq kG\fb$ and so $\fF_k(Y)\neq 
\{0\}$ by Property~(A) above. Consequently, by Corollary~\ref{othm2c},
we have $\fF_k(Y)=\fF_k(\St_k)=\langle \theta_{\fu_1\fe}\rangle_k$ and 
$\cH_k$ acts via $\varepsilon$. Furthermore, by the identification
$\fF_k(Y)\subseteq \fF_k(\St_k)$, we must have $\theta_{\fu_1\fe}(kG\fb)
\subseteq Y$ and so $\fu_1\fe\in Y$. \qed
\end{proof}

\begin{prop} \label{cmult} Let $Y\subseteq \St_k$ be as in 
Theorem~\ref{othm2}. Then $Y$ is absolutely irreducible and occurs 
only once as a composition factor of $\St_k$. Moreover, $Y$ is the 
only composition factor of $\St_k$ which belongs to $\Irr_k(G\mid B)$.
\end{prop}

\begin{proof} Recall that $\fF_k(Y)\neq \{0\}$ and $Y$ corresponds to 
$\varepsilon \colon \cH_k\rightarrow k$ via~($\spadesuit$). Hence, by 
\cite[2.13(d)]{my98}, we have $\End_{kG}(Y) \cong \End_{\cH_k}(\varepsilon)
\cong k$ and so $Y$ is absolutely irreducible. Now let $\{0\}=Z_0
\subsetneqq Z_1 \subsetneqq \ldots \subsetneqq Z_r=\St_k$ be a 
composition series and $Y_i:=Z_i/Z_{i-1}$ for $i=1,\ldots,r$ be the 
corresponding simple factors. Since $\ell\neq p$, the restriction 
of $\St_k$ to the subgroup $U\subseteq B$ is semisimple and, hence, 
isomorphic to the direct sum of the restrictions of the $Y_i$. 
Taking fixed points under $U$, we obtain
\[\dim \Fix_U(Y_1)+\ldots +\dim \Fix_U(Y_r)=\dim \Fix_U(\St_k)=\dim 
\langle \fu_1e\rangle_k=1.\]
Now, if $Y_i\in \Irr_k(G\mid B)$, then $\Fix_U(Y_i)\supseteq 
\Fix_B(Y_i) \cong \fF_k(Y_i)\neq \{0\}$ by Property~(A) and so we obtain
a non-zero contribution to the sum on the left hand side. Hence, there 
can be at most one $Y_i$ which belongs to $\Irr_k(G\mid B)$. Since $Y_1=Y
\subseteq kG\fb$ does belong to $\Irr_k(G\mid B)$, this proves the remaining 
assertions. \qed
\end{proof}

\begin{exmp} \label{soctriv} It is easily seen that $\fF_k(k_G)$ is also 
$1$-dimensional (spanned by the function $kG\fb\rightarrow k$ which takes 
constant value~$1$ on all $g\fb$ for $g\in G$) and $\cH_k$ acts on 
$\fF_k(k_G)$ via the algebra homomorphism $\mbox{ind}\colon \cH_k
\rightarrow k$ such that $\mbox{ind}(T_s)=q_s$ for all $s\in S$; see, 
e.g., \cite[4.3.3]{geja}. Let $Y$ be the simple socle of $\St_k$, as
in Theorem~\ref{othm2}. Then, by ($\spadesuit$), we obtain:
\[ Y\cong k_G\quad \Leftrightarrow\quad \fF_k(Y)\cong \fF_k(k_G)\quad
\Leftrightarrow\quad \varepsilon=\mbox{ind} \quad \Leftrightarrow\quad
q_s=-1 \mbox{ for all $s\in S$}.\]
Thus, we recover a result of Hiss \cite{hiss0} and Khammash \cite{kham} 
in this way. Furthermore, Proposition~\ref{cmult} implies that, if 
$q_s\neq 1$ for some $s\in S$, then $k_G$ is not even a composition 
factor of $\St_G$. (This result is also contained in \cite{hiss0}.)
\end{exmp}

\begin{lem} \label{stcomp} Let $M'$ be the Steinberg component in a 
given direct sum decomposition of $kG\fb$, as in Remark~\ref{othm2a}. 
Then $\St_k=M'$ if and only if $\St_k$ is irreducible. 
\end{lem}

\begin{proof} Assume first that $M'=\St_k$ and let $Z\subsetneqq \St_k$ be
a maximal submodule. Now $\St_k=M'$ is a factor module of $kG\fb$ and so 
$\St_k/Z$ belongs to $\Irr_k(G\mid B)$, by Property~(A). Hence,
by Proposition~\ref{cmult}, we must have $Z=\{0\}$. Conversely, assume 
that $\St_k$ is irreducible. Then $\ell\nmid [G:B]$ by Theorem~\ref{stein}.
If $\ell=p$, then $\fe$ is a non-zero scalar multiple of an idempotent 
in $kG$, by \cite[Lemma~2]{St57}. Hence, $\St_k$ is projective in this 
case and so $\St_k$ is a direct summand of $kG\fb$. If $\ell\neq p$, then 
the assumption $\ell\nmid [G:B]$ implies that $kG\fb$ is semisimple; 
see \cite[Lemma~4.3.2]{geja}. Hence, again, $\St_k$ is a direct summand 
of $kG\fb$. In both cases, it follows that $\St_k=M'$. \qed
\end{proof}

\begin{exmp} \label{soctrivrk1} Assume that $G$ has a $BN$-pair of 
rank~$1$, that is, $W=\langle s \rangle$ where $s\in W$ has order~$2$. 
Then, by the sharp form of the Bruhat decomposition, we have $[G:B]=1+|U|=
1+\dim \St_k$. Thus, there are only two cases.

If $q_s\neq -1$, then $kG\fb\cong k_G \oplus \St_k$ and $\St_k$ is 
irreducible by Theorem~\ref{stein}.

If $q_s=-1$, then the socle of $\St_k$ is the trivial module $k_G$
by Example~\ref{soctriv}.

\noindent In the second case, the structure of $\St_k$ can be quite 
complicated. For example, let $G={^2G}_2(q^2)$ be a Ree group, where 
$q$ is an odd power of $\sqrt{3}$. If $\ell=2$, then Landrock--Michler 
\cite[Prop.~3.8(b)]{LaMi} determined socle series for $kG\fb$ and $\St_k$:
\[ kG\fb:\quad \begin{array}{ccc} & k_G &\\& \varphi_2 & \\ \varphi_4&  
\varphi_3& \varphi_5\\ & \varphi_2 & \\ & k_G& \end{array}\qquad\qquad\qquad
\St_k:\quad \begin{array}{ccc} & \varphi_2 & \\ \varphi_4&  
\varphi_3& \varphi_5\\ & \varphi_2 & \\ & k_G& \end{array}\]
where $\varphi_i$ ($i=1,2,3,4,5$) are simple $kG$-modules and $\varphi_4$ 
is the contragredient dual of $\varphi_5$. 

It is not true in general that $\St_k$ has a unique simple quotient.
For example, let $G=\GU_3(q)$ where $q$ is any prime power. Assume 
that $\ell$ is a prime such that $\ell\mid q+1$. Then socle series 
for $\St_k$ are known by the work of various authors; see Hiss 
\cite[Theorem~4.1]{hiss} and the references there: 
\[ \St_k:\quad \begin{array}{ccc}  & & \varphi\\ 
\varphi \oplus \vartheta & \qquad\qquad & \vartheta\\ k_G & & \varphi \\
& & k_G \\[1pt] \mbox{\scriptsize ($\ell=2$ and $4\mid q-1$)} & & 
\mbox{\scriptsize (otherwise)}  \end{array}\]
where $\varphi$ and $\vartheta$ are simple $kG$-modules. See also the 
examples in Gow \cite[\S 5]{gow}.
\end{exmp}

\begin{exmp} \label{socgln} Let $G=\GL_n(q)$ and $\cU_k(G)$ be
the set of all $Y\in \Irr_k(G)$ such that $Y$ is a composition factor
of $kG\fb$. James \cite[16.4]{Jam1} called these the {\em unipotent 
modules} of $G$ and showed that there is a canonical parametrization
\[ \cU_k(G)=\{D_\mu\mid\mu\vdash n\}.\] 
(See also \cite[7.35]{Jam3}.) Here, $D_{(n)}=k_G$, as follows immediately 
from \cite[Def.~1.11]{Jam1}. 

The above parametrisation is characterised as follows. For each 
partition $\lambda\vdash n$, let $M_\lambda$ be the permutation 
representation of $G$ on the cosets of the corresponding parabolic 
subgroup $P_\lambda \subseteq G$ (block triangular matrices with diagonal
blocks of sizes given by the parts of $\lambda$). Then $D_\mu$ has
composition multiplicity~$1$ in $M_\mu$ and composition multipliciy $0$ 
in $M_\lambda$ unless $\lambda \trianglelefteq \mu$; see \cite[11.12(iv), 
11.13]{Jam1}. This shows, in particular, that the above parametrisation is 
consistent with other known parametrisations of $\cU_k(G)$, e.g., 
the one in \cite[\S 3]{lymgh} based on properties of the $\ell$-modular 
decomposition matrix of $G$.

If $\ell=0$, let us set $e:=\infty$; if $\ell$ is a prime ($\neq p$), 
then let 
\[e:=\min \{i\geq 2\mid 1+q+q^2+\cdots + q^{i-1}\equiv 0 \bmod \ell\}.\] 
Then, by James \cite[Theorem~8.1(ix), (xi)]{Jam3}, the subset 
$\Irr_k(G \mid B) \subseteq \cU_k (G)$ consists precisely of those 
$D_\lambda$ where $\lambda \vdash n$ is $e$-regular.
Now let $Y$ be the socle of $\St_k$, as in Theorem~\ref{othm2}. Then 
$Y\in \Irr_k(G\mid B)$ and so $Y\cong D_{\mu_0}$ for a well-defined 
$e$-regular partition $\mu_0\vdash n$. This partition $\mu_0$ can be 
identified as follows. Write $n=(e-1)m+r$ where $0\leq r<e
-1$. (If $e=\infty$, then $m=0$ and $r=n$.) We claim that
\[\mu_0=(\underbrace{m+1,m+1,\ldots,m+1}_{\text{$r$ times}},
\underbrace{m,m,\ldots,m}_{\text{$e{-}r{-}1$ times}})\vdash n.\]
Indeed, by Theorem~\ref{othm2} and ($\spadesuit$), the $kG$-module $Y$ 
corresponds to the $1$-dimensional representation $\varepsilon\colon \cH_k
\rightarrow k$. Now $\Irr(\cH_k)$ also has a natural parametrisation by the 
$e$-regular partitions of $n$, a result originally due to Dipper 
and James; see, e.g., \cite[8.1(i)]{Jam3}, \cite[\S 3.5]{geja} and the 
references there. By \cite[Theorem~8.1(xii)]{Jam3} (or the general
discussion in \ref{princ} below), this parametrisation is compatible with 
the above parametrisation of $\cU_k(G)$, in the sense that the 
partition $\mu_0\vdash n$ such that $Y\cong D_{\mu_0}$ must also 
parametrise~$\varepsilon$. Now note that $\varepsilon\circ\gamma =
\mbox{ind}$, where $\mbox{ind} \colon \cH_k \rightarrow k$ is defined in 
Example~\ref{soctriv} and $\gamma \colon \cH_k\rightarrow \cH_k$ is the 
algebra automorphism such that $\gamma (T_s)= -q_sT_s^{-1}$ (see 
\cite[Exc.~8.2]{gepf}). The definitions immediately show that $\mbox{ind}$ 
is parametrised by the partition $(n)$. Thus, our problem is a special 
case of describing the ``Mullineux involution'' on $e$-regular 
partitions which, for the particular partition $(n)$, has the solution 
stated above by Mathas \cite[6.43(iii)]{mathas}. (I~thank Nicolas Jacon 
for pointing out this reference to me.)

We remark that Ackermann \cite[Prop.~3.1]{Ack} already showed that
$\St_k$ has precisely one composition factor $D_{\mu_0}$ where $\mu_0$
is the image of $(1^n)$ under the Mullineux involution; however, he 
does not locate $D_{\mu_0}$ in a composition series of $\St_k$.
\end{exmp}

\begin{abs} \label{princ} For general $G$, the definition of 
{\em unipotent modules} is more complicated than for $\GL_n(q)$ (see, 
e.g., \cite[\S 1]{ghm1}), but one can still proceed as follows. Let us 
assume that $G$ is a true finite group of Lie type, as in 
Remark~\ref{truegf}. We shall write $\Irr_\C(W)=
\{E^\lambda \mid \lambda \in\Lambda\}$ where $\Lambda$ is some finite 
indexing set. It is a classical fact that, if $k=\C$, then there is a 
bijection $\Irr_\C(W) \leftrightarrow \Irr(\cH_\C)$, $E^\lambda 
\leftrightarrow E^\lambda_q$, and a decomposition 
\[ \C G\fb\cong \bigoplus_{\lambda \in \Lambda} \underbrace{\rho^\lambda
\oplus \ldots \oplus \rho^\lambda}_{\text{$\dim E^\lambda$ times}}
\qquad \mbox{where}\qquad \fF_\C(\rho^\lambda)\cong E_q^\lambda \quad
\mbox{for all $\lambda \in \Lambda$}.\]
Hence, we have a natural parametrisation $\Irr_\C(G\mid B)=\{\rho^\lambda
\mid \lambda \in \Lambda\}$ in this case; see, e.g., Carter 
\cite[\S 10.11]{Ca2}, Curtis--Reiner \cite[\S 68B]{CR2} (and also 
\cite[Exp.~2.2]{my98}, where the Hom functor is linked to the settings 
in [{\em loc.\ cit.}]). In general, under some mild conditions on $k$, 
it is shown in \cite[Theorem~1.1]{myprinc} that there is still a natural 
parametrisation of $\Irr_k(G\mid B)$, but now by a certain subset 
$\Lambda_k^\circ \subseteq \Lambda$. We briefly describe how this is 
done, where we refer to the exposition in \cite[\S 4.4]{geja} for further 
details and references.

First, to each $E^\lambda$ one can attach a numerical value 
$\ba_\lambda\in \Z_{\geq 0}$ (Lusztig's ``$\ba$-invariant''); note that 
$\lambda \mapsto\ba_\lambda$ depends on the exponents $c_s$ such that 
$|U_s|=q^{c_s}$ for $s\in S$. Then, under some mild conditions on 
$k$, the algebra $\cH_k$ is ``cellular'' in the sense of Graham--Lehrer,
where the corresponding cell modules are parametrized by $\Lambda$, and 
$\Lambda$ is endowed with the partial order $\preceq$ such that $\mu \preceq
\lambda$ if and only if $\mu=\lambda$ or $\ba_\lambda<\ba_\mu$. Finally,
by the general theory of cellular algebras, there is a canonically defined 
subset $\Lambda_k^\circ \subseteq \Lambda$ such that 
\[  \Irr(\cH_k)=\{L_k^\mu \mid \mu\in \Lambda_k^\circ\},\]
where $L_k^\lambda$ is the unique simple quotient of the cell module
corresponding to $\lambda\in\Lambda_k^\circ$. Hence, via the Hom 
functor and ($\spadesuit$), we obtain the desired parametrisation
\[ \Irr_k(G\mid B)=\{Y^\mu \mid \mu\in\Lambda_k^\circ\} \qquad
\mbox{where} \qquad \fF_k(Y^\mu)\cong L_k^\mu \mbox{ for $\mu
\in\Lambda_k^\circ$}.\]
Let $M\in \Irr(\cH_k)$ and denote by $d_{\lambda,M}$ the multiplicity of 
$M$ as a composition factor of the cell module indexed by~$\lambda\in 
\Lambda$. Then, by \cite[3.2.7]{geja}, the unique $\mu\in \Lambda_k^\circ$ 
such that $M\cong L^\mu$ is characterised by the condition that $\mu$ is 
the unique element at which the function $\{\lambda \in \Lambda \mid 
d_{\lambda,M}\neq 0\} \rightarrow \ba_\lambda$ takes its minimum value.

Now recall that the simple socle $Y\subseteq \St_k$ belongs to $\Irr_k(G 
\mid B)$ and it corresponds, via the Hom functor and ($\spadesuit$), to the 
$1$-dimensional representation $\varepsilon \colon \cH_k\rightarrow k$. 
The unique $\mu_0\in \Lambda_k^\circ$ such that $Y\cong Y^{\mu_0}$ is 
found as follows. We order the elements of $\Lambda$ according to 
increasing $\ba$-invariant; then $\mu_0$ is the first element in this list 
for which we have $d_{\mu_0,\varepsilon} \neq 0$. Note also that 
$\varepsilon$ is afforded by a cell module; the unique $\lambda_0 \in
\Lambda$ labelling this cell module is characterised by the condition
that $\ba_{\lambda_0}=\max\{\ba_\lambda \mid \lambda \in \Lambda\}$
(see, e.g., \cite[1.3.3]{geja}).

For example, if $G=\GL_n(q)$, then $W=\fS_n$ and $\Lambda$ is the set of 
partitions of $n$. In this case, we have $\lambda_0=(1^n)$ and $\mu_0$ is
described in Example~\ref{socgln}.

If tables with the decomposition numbers $d_{\lambda,M}$ for $\cH_k$ are 
known, then $\mu_0$ can be simply read off these tables. Thus, $\mu_0
\in \Lambda_k^\circ$ can be determined for all groups of exceptional type, 
using the information in \cite[App.~F]{gepf}, \cite[Chap.~7]{geja}; the
results are given in Table~\ref{tabsoc}. (If there is no entry in this
table corresponding to a certain value of $e$, then this means that 
$\ell \nmid [G:B]$ and so $\St_k$ is simple.)
\end{abs}

\begin{table}[htbp]
\caption{The labels $\mu_0\in \Lambda_k^\circ$ for $G$ of exceptional 
type and $\ell \mid [G:B]$} \label{tabsoc} 
\centerline{\small $\renewcommand{\arraystretch}{1.1} \arraycolsep 3pt 
\begin{array}{ccccccccccccccccc}  \hline
e & 2 & 3 & 4 & 5 & 6 & 7 & 8 & 9 & 10 & 12 & 14 & 15 & 18 & 20 & 24 & 30\\
\hline G_2(q) & 1_W & \sigma_2& & &\sigma_1 &  & & & & & & & & & &  \\
{^3\!D}_4(q) & 1_W& \sigma_2& & &\varepsilon_1 &  & & & & \sigma_1 &&&&&&\\
{^2\!F}_4(q^2)&\sigma_2&&\varepsilon_1& &\sigma_2&& & & & \sigma_1 &&&&&&\\ 
F_4(q)&1_1&4_1 &6_1 & &12 &  & 9_4& & & 4_5 & & & & & &  \\
{^2\!E}_6(q) &8_1& 4_1&1_3 && 4_4& &8_4 & &8_2 &9_4 & & &4_5 & & &\\
E_6(q) &1_p &15_q &10_s &24_p'& 60_p'& &30_p' & 20_p'& &6_p' & & & & & &\\
E_7(q) &1_a& 15_a' & 70_a'&84_a' & 210_b'& 27_a' & 105_b'& 35_b'& 21_b
&56_a &27_a' & & 7_a& & &  \\
E_8(q) &1_x & 50_x & 175_x&168_y & 420_y&  300_x' & 2835_x'& 50_x'& 
448_z'& 1400_x' & 700_x' & 84_x' & 210_x' & 112_z' & 35_x' & 8_z' \\
\hline \multicolumn{17}{l}{\text{In type $E_8$, $\ell>5$; otherwise, 
$\ell>3$; here, $e:=\min\{i\geq 2\mid 1+q_0+q_0^2+\cdots + q_0^{i-1}\equiv
0 \bmod\ell\}$,}}\\ \multicolumn{17}{l}{\text{with $q_0:=q$ in all cases
except for ${^2\!F}_4(q^2)$, where $q$ is an odd power of $\sqrt{2}$ and 
$q_0:=q^2$.}} \end{array}$}
\end{table}

Just to illustrate the procedure (and to fix some notation), let us
consider the case where $G={^2\!F}_4(q^2)$; here, $q$ is an odd power
of $\sqrt{2}$. Setting $q_0:=q^2$, we have 
\[ |B|=q_0^{12}(q_0-1)^2 \qquad \mbox{and}\qquad 
[G:B]=(q_0+1)(q_0^2+1)(q_0^3+1) (q_0^6+1).\]
Now, $W=\langle s_1,s_2\rangle$ is dihedral of order $16$ and we have
$\{q_{s_1},q_{s_2}\}=\{q_0,q_0^2\}$. We fix the notation such that 
$q_{s_1}=q_0^2$ and $q_{s_2}=q_0$. As in \cite[Exp.~1.3.7]{geja}, we have
\[ \Irr_\C(W)=\{1_W,\varepsilon,\varepsilon_1,\varepsilon_2,\sigma_1,
\sigma_2,\sigma_3\}\]
where $\varepsilon_1$, $\varepsilon_2$ are $1$-dimensional and
determined by $\varepsilon_1(s_1)=\varepsilon_2(s_2)=1$ and 
$\varepsilon_1(s_2)=\varepsilon_2(s_1)=-1$; furthermore, each $\sigma_i$
is $2$-dimensional and the labelling is such that the trace of $\sigma_1$
on $s_1s_2$ equals $\sqrt{2}$, that of $\sigma_2$ equals $0$, and that of
$\sigma_3$ equals $-\sqrt{2}$.

The ``mild condition'' on $k$ is that $\ell$ must be a ``good'' prime for 
the underlying algebraic group; so, $\ell>3$. Assuming 
that $\ell\mid [G:B]$, we have the following cases to consider:
\[ \ell \mid q_0+1,\qquad \ell \mid q_0^2+1, \qquad \ell\mid q_0^2-q_0+1,
\qquad \ell \mid q_0^4-q_0^2+1,\]
which correspond to $e=2,4,6,12$, respectively. For example, for $e=2,4$, 
the decomposition numbers $d_{\lambda,M}$ are given as follows; see 
\cite[Table~7.6]{geja}:
\begin{center} 
{\small $\arraycolsep 5pt \renewcommand{\arraystretch}{1.1}
\begin{array}{cc|cccc} \hline (e=2) & \ba_\lambda & 
\multicolumn{4}{c}{ d_{\lambda,M}} \\ \hline
\bullet\; 1_W         &  0 & 1 & . & . & .\\
\varepsilon_1         &  1 & 1 & . & . & .\\
\bullet\;\sigma_1     &  2 & . & . & 1 & .\\
\bullet\;\sigma_2     &  2 & 1 & 1 & . & .\\
\bullet\;\sigma_3     &  2 & . & . & . & 1\\
  \varepsilon_2       &  5 & . & 1 & . & .\\
   \varepsilon        & 12 & . & 1 & . & .\\
\hline \end{array} \qquad \qquad \qquad 
\begin{array}{cc|cccc} \hline (e=4) & \ba_\lambda &  
\multicolumn{4}{c}{ d_{\lambda,M}} \\ \hline
\bullet\; 1_W             &  0 & 1 & . & . & . \\
\bullet\; \varepsilon_1   &  1 & . & 1 & . & . \\
\bullet\;   \sigma_1      &  2 & . & . & 1 & . \\
\bullet\;   \sigma_2      &  2 & . & . & . & 1 \\
            \sigma_3      &  2 & 1 & 1 & . & . \\
            \varepsilon_2 &  5 & 1 & . & . & . \\
            \varepsilon   & 12 & . & 1 & . & . \\
\hline \end{array}$}
\end{center}
Those representations which belong to $\Lambda_k^\circ$ are marked by 
``$\bullet$''. The above procedure for finding $\mu_0$ now yields
$\mu_0=\sigma_2$ for $e=2$ and $\mu_0=\varepsilon_1$ for $e=4$. 

\begin{rem} \label{socclass}
For groups of classical type, $\Lambda$ is a set of certain bipartitions 
of $n$ and the subsets $\Lambda_k^\circ\subseteq \Lambda$ are explicitly
known in all cases; see \cite{geja}. Nicolas Jacon has pointed out to me 
that then $\mu_0\in \Lambda_k^\circ$ can be described explicitly using the 
results of Jacon--Lecouvey \cite{JaLe}. This will be discussed elsewhere 
in more detail.
\end{rem}

%%%%%%%%%%%%%%%%%%%%%%%%%%%%%%%%%%%%%%%%%%%%%%%%%%%%%%%%%%%%%%%%%%%%%%%
\section{The Steinberg module and Harish-Chandra series} \label{sec2b}

We shall assume from now on that $G=\bG^F$ is a true finite group of 
Lie type, as in Remark~\ref{truegf}. Then $G$ satisfies the ``commutator
relations'' and so the parabolic subgroups of $G$ admit Levi decompositions;
see Carter \cite[\S 2.6]{Ca2}, Curtis--Reiner \cite[\S 70A]{CR2}. For each 
subset $J\subseteq S$, let $P_J\subseteq G$ be the corresponding standard 
parabolic subgroup, with Levi decomposition $P_J=U_J\rtimes L_J$ where 
$U_J$ is the unipotent radical of $P_J$ and $L_J$ is the standard Levi 
complement. The Weyl group of $L_J$ is $W_J=\langle J\rangle$ and $L_J$ 
itself is a true finite group of Lie type. Let $A$ be a commutative ring 
(with $1_A$) such that $p$ is invertible in $A$. Then we obtain functors,
called Harish-Chandra induction and restriction, 
\begin{align*}
R_J^S\colon AL_J\mbox{-modules}&\rightarrow AG\mbox{-modules},\\
\Rst_J^S\colon AG\mbox{-modules}&\rightarrow AL_J\mbox{-modules},
\end{align*}
which satisfy properties analogous to the usual induction and restriction, 
like transitivity, adjointness and a Mackey formula; we refer to 
\cite{DipDu}, \cite{hiss2}  and the survey in \cite[\S 3]{my98} for further 
details. An $AG$-module $Y$ is 
called {\em cuspidal} if $\Rst_J^S(Y)=\{0\}$ for all $J\subsetneqq S$. In 
this general setting, we have the following important result due to 
Dipper--Du \cite{DipDu} and Howlett--Lehrer \cite{HL2}. Let $I,J\subseteq 
S$ be subsets such that $wIw^{-1}=J$ for some $w\in W$; let $n\in N$ be a 
representative of~$w$. Then $nL_In^{-1} =L_J$ and 
\[ R_I^S(X)\cong R_J^S({^n\!X})\qquad \mbox{for any $A L_I$-module $X$}.\]
(Here, we denote by ${^n\!X}$ the usual conjugate module for $\cO L_J$; see, 
e.g., \cite[\S 10B]{CR1}). In analogy to \cite[(70.11)]{CR2}, we will 
refer to this as the ``Strong Conjugacy Theorem''.

Furthermore, we now place ourselves in the usual setting for modular 
representation theory; see, e.g., \cite[\S 16A]{CR1}. 
Thus, we assume that our field $k$ has characteristic $\ell>0$ (where 
$\ell\neq p$ as before), and that $k$ is the residue field of a discrete 
valuation ring $\cO$ with field of fractions $K$ of characteristic~$0$. 
Both $K$ and $k$ will be assumed to be ``large enough'', that is, $K$ 
and $k$ are splitting fields for $G$ and all its subgroups. An 
$\cO G$-module $M$ which is finitely generated and free over $\cO$ will 
be called an $\cO G$-lattice. If $M$ is an $\cO G$-lattice, then we may 
naturally regard $M$ as a subset of the $KG$-module $KM:=K\otimes_\cO M$; 
furthermore, by ``$\ell$-modular reduction'', we obtain a $kG$-module 
$\overline{M}:=M/\fp M\cong k \otimes_\cO M$ where $\fp$ is the unique 
maximal ideal of $\cO$. Finally note that, by \cite[Exc.~6.16]{CR1}, 
idempotents can be lifted from $kG$ to $\cO G$, hence, $\cO G$ is 
``semiperfect''. We shall freely use standard notions and properties of 
projective covers, pure submodules etc.; see \cite[\S 4D, \S 6]{CR1}. 

Harish-Chandra induction and restriction are compatible with this set-up.
Indeed, let $J\subseteq S$. If $X$ is an $\cO L_J$-lattice and $Y$ is an 
$\cO G$-lattice, then $R_J^S(X)$ is an $\cO G$-lattice, $\Rst_J^S(Y)$ 
is an $\cO L_J$-lattice, and we have 
\begin{alignat*}{3}
KR_J^S(X)&\cong R_J^S(KX)&\qquad &\mbox{and}&\qquad 
K\Rst_J^S(Y) &\cong \Rst_J^S(KY),\\ \overline{R_J^S(X)}&\cong R_J^S
(\overline{X}) &\qquad &\mbox{and} &\qquad \overline{\Rst_J^S(Y)} 
&\cong \Rst_J^S(\overline{Y}).
\end{alignat*}

\begin{abs} \label{twohiss1}
By Theorem~\ref{stein}, we have the ``canonical'' Steinberg lattice 
$\St_\cO=\cO G\fe$. Here, we can naturally identify $K\St_\cO$ with
$\St_K$ and $\overline{\St}_\cO$ with $\St_k$. Since $\charak(K)=0$, 
the $KG$-module $\St_K$ is irreducible. We obtain further 
$\cO G$-lattices affording $\St_K$ as follows. Let $\sigma \colon 
U\rightarrow K^\times$ be a group homomorphism. Since $\ell\nmid |U|$, 
the values of $\sigma$ lie in $\cO^\times$. Then $\fu_\sigma := 
\sum_{u\in U} \sigma(u)u \in \cO G$ and so $\Gamma_\sigma :=\cO G
\fu_\sigma$ is an $\cO G$-lattice. Since $\fu_\sigma^2=|U|\fu_\sigma$ 
and $|U|$ is a unit in $\cO$, the module $\Gamma_\sigma$ is projective. 
Furthermore, $\Hom_{\cO G}(\Gamma_\sigma,\St_\cO) \cong \fu_\sigma 
\St_\cO=\langle \fu_\sigma \fe \rangle_\cO$ where the last equality 
holds since $\{u\fe\mid u \in U\}$ is an $\cO$-basis of $\St_\cO$ and 
since $\fu_\sigma u=\sigma(u)^{-1} \fu_\sigma$ for all $u\in U$. Thus, 
\[\Hom_{\cO G}(\Gamma_\sigma, \St_\cO)=\langle\varphi_\sigma \rangle_{\cO}
\qquad \mbox{where}\qquad \varphi_\sigma\colon \Gamma_\sigma\rightarrow 
\St_\cO, \;\; \gamma \mapsto \gamma \fu_\sigma \fe.\]
The same computation also works over $K$, hence we obtain
$\dim \Hom_{KG}(K\Gamma_\sigma,\St_K)=1$.  
\end{abs}

\begin{prop}[Cf.\ Hiss \protect{\cite[\S 6]{hiss2}}] \label{twohiss} 
For any $\sigma\colon U\rightarrow K^\times$ as above, there is a unique 
pure $\cO G$-sublattice $\Gamma_\sigma' \subseteq \Gamma_\sigma$ such that, 
if we set $\cS_\sigma:=\Gamma_\sigma/\Gamma_{\sigma}'$, then $K\cS_\sigma
\cong\St_K$. Furthermore, $\varphi_\sigma$ induces an injective 
$\cO G$-module homomorphism $\rho_\sigma \colon \cS_\sigma \rightarrow 
\St_\cO$. The $kG$-module $D_\sigma:=\overline{\cS}_\sigma/
\rad(\overline{\cS}_\sigma)$ is simple and it occurs exactly once as a 
composition factor of $\St_k$.
\end{prop}

\begin{proof} First we show that a sublattice $\Gamma_\sigma'\subseteq 
\Gamma_\sigma$ with the desired properties exists. Now, since $KG$ is 
semisimple and $\dim \Hom_{KG}(K\Gamma_\sigma,\St_K)=1$, we can write 
$K\Gamma_\sigma =Z_1 \oplus Z_2$ where $Z_1$, $Z_2$ are $KG$-submodules 
such that $Z_1 \cong \St_K$ and $\Hom_{KG}(Z_2,\St_K)=\{0\}$. Then 
$\Gamma_\sigma':=\Gamma_\sigma\cap Z_2$ is a pure submodule of 
$\Gamma_\sigma$. Consequently, $\cS_\sigma:=\Gamma_\sigma/\Gamma_\sigma'$ 
is an $\cO G$-lattice such that $K\cS_\sigma \cong \St_K$. Now consider 
the map $\varphi_\sigma \colon \Gamma_\sigma \rightarrow \St_\cO$. Since
$\Hom_{KG}(Z_2,\St_K)=\{0\}$, we have $\Gamma_\sigma' \subseteq 
\ker(\varphi_\sigma)$ and so we obtain an induced $\cO G$-module 
homomorphism $\rho_\sigma \colon \cS_\sigma \rightarrow \St_\cO$. Since 
$K\cS_\sigma$ is irreducible and $\varphi_\sigma \neq 0$, the map
$\rho_\sigma$ is injective. 

Let us further write $\Gamma_\sigma=P_1\oplus \ldots \oplus P_r$ where 
each $P_i$ is an $\cO G$-lattice which is projective and indecomposable. 
Then $K\Gamma_\sigma=KP_1\oplus \ldots \oplus KP_r$. Since $\dim \Hom_{KG}
(K\Gamma_\sigma,\St_K)=1$, there is a unique $i$ such that $\Hom_{KG}
(KP_i,\St_K)\neq \{0\}$. Then $Z_1 \subseteq KP_i$ and $KP_j\subseteq 
Z_2$ for all $j\neq i$. Thus, we have $\cS_\sigma \cong P_i/(P_i\cap Z_2)$ 
and so $P_i$ is a projective cover of $\cS_\sigma$. This certainly 
implies that $D_\sigma=\overline{\cS}_\sigma/\rad(\overline{\cS}_\sigma) 
\cong \overline{P}_i/\rad(\overline{P}_i)$ is simple and that 
$\overline{P}_i$ is a projective cover of $D_\sigma$. Since $\fu_\sigma \fe
\in \rho_\sigma (\cS_\sigma)$, the induced map $\overline{\rho}_\sigma 
\colon\overline{\cS}_\sigma\rightarrow \St_k$ is non-zero and so $D_\sigma$
is a composition factor of $\St_k$. On the other hand, by standard 
properties of projective modules, the composition multiplicity of 
$D_\sigma$ in $\St_k$ is bounded above by 
\[\dim \Hom_{kG}(P_i,\St_k)\leq \dim \Hom_{kG}(\overline{\Gamma}_\sigma,
\St_k)=\Hom_{KG}(K\Gamma_\sigma,\St_K)=1.\]
Once the existence of $\Gamma_\sigma'$ is shown, the uniqueness 
automatically follows since the intersection of pure submodules is 
pure and $\St_K$ has multiplicity~$1$ in $K\Gamma_\sigma$. \qed
\end{proof}

\begin{abs} \label{gggr} We assume from now on that the center of $\bG$
is connected. Furthermore, we shall fix a group homomorphism $\sigma
\colon U\rightarrow K^\times$ which is a {\em regular character}, that is,
we have $U^*\subseteq \ker(\sigma)$ and the restriction of $\sigma$ to 
$U_s$ is non-trivial for all $s\in S$. (Such characters always exist.) 
Then the corresponding module $\Gamma_\sigma=\cO G\fu_\sigma$ is called 
a {\em Gelfand-Graev module} for $G$; see \cite[\S 8.1]{Ca2} or 
\cite[p.~258]{lecst}. Since the center of $\bG$ is assumed to be connected,
all regular characters of $U$ are conjugate under the action of $H$
and, hence, the corresponding Gelfand-Graev modules will all be 
isomorphic; see \cite[8.1.2]{Ca2}. 

For any $J\subseteq S$, we have $L_J=\bL^F$ where $\bL$ is an $F$-stable 
Levi subgroup in $\bG$; here, the center of $\bL$ will also be connected; 
see \cite[8.1.4]{Ca2}. Our regular character $\sigma$ determines a 
regular character of $L_J$; see, e.g., \cite[8.1.6]{Ca2}. Hence, we also 
have a well-defined Gelfand-Graev module for $\cO L_J$, which we denote by 
$\Gamma_{\sigma}^J$. Applying the construction in Proposition~\ref{twohiss},
we obtain an $\cO L_J$-lattice $\cS_\sigma^J=\Gamma_\sigma^J/
(\Gamma_\sigma^J)'$ and a simple $kL_J$-module $D_\sigma^J:=
\overline{\cS}_\sigma^J/\rad(\overline{\cS}_\sigma^J)$. We have 
$K\cS_\sigma^J\cong \St_K^J$, the Steinberg module for $KL_J$. 
\end{abs}

\begin{lem} \label{sthc0} Let $J\subseteq S$. Then the following hold.
\begin{itemize}
\item[{\rm (i)}] We have $\Rst_J^S(\Gamma_{\sigma})\cong 
\Gamma_{\sigma}^J$ and $\Rst_J^S(\cS_\sigma)\cong \cS_\sigma^J$ (as 
$\cO L_J$-modules).
\item[{\rm (ii)}] If $I\subseteq S$ and $n\in N$ are such that 
$nL_In^{-1}=L_J$, then ${^n\!\cS}_{\sigma}^I\cong \cS_{\sigma}^{J}$
(as $\cO L_J$-modules) and ${^n\!D}_{\sigma}^I \cong D_{\sigma}^{J}$
(as $kL_J$-modules).
\end{itemize}
\end{lem}

\begin{proof} (i) By a result of Rodier (\cite[8.1.5]{Ca2}), 
we have $\Rst_J^S(K\Gamma_{\sigma})\cong K\Gamma_{\sigma}^J$; by 
\cite[(71.6)]{CR2}, we also have $\Rst_J^S(\St_K)\cong \St_K^J$. So (i)
follows by a standard argument;  see, e.g., \cite[5.14, 5.15]{my98}.

(ii) Since $\Rst_J^S(K\Gamma_{\sigma})\cong K\Gamma_{\sigma}^J$, it
is straightforward to show that ${^n\!K\Gamma}_{\sigma}^I \cong 
K\Gamma_{\sigma}^J$, using the ``Strong Conjugacy Theorem''. So we 
also have ${^n\Gamma}_{\sigma}^I \cong \Gamma_{\sigma}^J$ as 
$\cO L_J$-modules (since these modules are projective). This then implies 
(ii) by the construction of $\cS_{\sigma}^J$ and $D_{\sigma}^J$.\qed
\end{proof}

\begin{abs}  \label{sthc1a} Let $\cP_\sigma^*$ be the set of all subsets
$J\subseteq S$ such that $D_{\sigma}^J$ is a cuspidal $kL_J$-module. 
For $J \in \cP_\sigma^*$, we denote by $\Irr_k(G\mid J,\sigma)$ the set of 
all $Y \in \Irr_k(G)$ such that $Y$ is isomorphic to a submodule of 
$R_J^S(D_{\sigma}^J)$. By the ``Strong Conjugacy Theorem'', this is a 
Harish-Chandra series as defined by Hiss \cite{hiss2}. Hence, by 
\cite[Theorem~5.8]{hiss2} (see also \cite[\S 3]{my98}), every simple 
submodule of $R_J^S(D_{\sigma}^J)$ is isomorphic to a factor module of 
$R_J^S(D_{\sigma}^J)$, and vice versa. Furthermore, using also 
Lemma~\ref{sthc0}(ii), we have for any $I,J\in \cP_\sigma^*$:
\begin{alignat*}{2}
\Irr_k(G\mid I,\sigma)&=\Irr_k(G\mid J,\sigma) &\qquad &\mbox{if
$J=wIw^{-1}$ for some $w\in W$},\\
\Irr_k(G\mid I,\sigma) &\cap \Irr_k(G\mid J,\sigma) =\varnothing 
& \qquad &\mbox{otherwise}.
\end{alignat*}
\end{abs}

Thus, having fixed a regular character $\sigma\colon U\rightarrow K^\times$
as in \ref{gggr}, the above constructions produce composition factors of 
$kG\fb$ arising from subsets $J\subseteq S$. The following two results are 
adaptations of Dipper--Gruber \cite[Cor.~2.24 and 2.40]{digr} to the 
present setting. 

\begin{prop} \label{sthc2} Let $J\in \cP_\sigma^*$. Then $\St_k$ has a unique 
composition factor which belongs to the series $\Irr_k(G\mid J,\sigma)$. 
\end{prop}

\begin{proof} First note that $\St_K\cong K\cS_\sigma$ and so, by
a classical result of Brauer--Nesbitt (see \cite[(16.16)]{CR1}), $\St_k$ 
and $\overline{\cS}_\sigma$ have the same composition factors (counting 
multiplicities). Using Lemma~\ref{sthc0}(i) and adjointness, we obtain 
\[\Hom_{kG}\bigl(\overline{\cS}_\sigma,R_J^S(D_{\sigma}^J)\bigr)\cong
\Hom_{kL_J}\bigl(\Rst_J^S(\overline{\cS}_\sigma),D_{\sigma}^J\bigr)\cong
\Hom_{kL_J}(\overline{\cS}_{\sigma}^J,D_{\sigma}^J)\neq \{0\}.\]
Hence, some simple submodule of $R_J^S(D_{\sigma}^J)$ will be 
isomorphic to a composition factor of $\overline{\cS}_\sigma$ and so the 
latter module has at least one composition factor which belongs to  
$\Irr_k(G\mid J,\sigma)$. On the other hand, since $D_{\sigma}^J$ is 
a quotient of $\overline{\Gamma}_{\sigma}^J$, there exists a surjective
$kG$-module homomorphism $R_J^S(\overline{\Gamma}_{\sigma}^J) \rightarrow 
R_J^S(D_{\sigma}^J)$. Now $R_J^S(\overline{\Gamma}_{\sigma}^J)$ is 
projective (see, e.g., \cite[3.4]{my98}) and every simple module in 
$\Irr_k(G\mid J,\sigma)$ also is a quotient of $R_J^S(D_{\sigma}^J)$ (see 
\ref{sthc1a}). Hence, by standard results on projective modules, the 
total number of composition factors (counting multiplicities) of 
$\overline{\cS}_\sigma$ which belong to $\Irr_k(G\mid J,\sigma)$ is 
bounded above by 
\[\dim \Hom_{kG}\bigl(R_J^S(\overline{\Gamma}_{\sigma}^J), 
\overline{\cS}_\sigma\bigr)=\dim \Hom_{KG}\bigl(R_J^S(K\Gamma_{\sigma}^J),
K\cS_\sigma\bigr).\]
Using Lemma~\ref{sthc0}(i) and adjointness, the dimension on the right 
hand side evaluates to $\dim \Hom_{KL_J}(K\Gamma_{\sigma}^J,
K\cS_{\sigma}^J)$, which is one by \ref{twohiss1}. \qed
\end{proof}

\begin{prop} \label{sthc3} Assume that every composition factor of $kG\fb$ 
belongs to $\Irr_k(G\mid J,\sigma)$ for some $J\in \cP_\sigma^*$. Then the 
following hold.
\begin{itemize}
\item[{\rm (i)}] $\St_k$ is multiplicity-free and the length of a
composition series of $\St_k$ is equal to the number of $J\in \cP_\sigma^*$ 
(up to $W$-conjugacy).
\item[{\rm (ii)}] The induced map $\overline{\rho}_\sigma\colon 
\overline{\cS}_\sigma\rightarrow \St_k$ is an isomorphism and so
$\St_k/\rad(\St_k)\cong D_\sigma$. 
\item[{\rm (iii)}] All composition factors of $\rad(\St_k)$ are non-cuspidal.
\end{itemize}
\end{prop}

\begin{proof} (i) Since $\St_k\subseteq kG\fb$, the hypothesis applies, 
in particular, to the composition factors of $\St_k$. It remains to 
use Proposition~\ref{sthc2}.

(ii) By the proof of Proposition~\ref{twohiss}, we have 
$\overline{\rho}_\sigma\neq 0$. Hence, it is enough to show that
$\overline{\rho}_\sigma$ is injective. By \cite[Theorem~5.16]{my98}, this
is equivalent to the following statement.
\begin{itemize}
\item[($\dagger$)] If $I\subseteq S$ is such that $\overline{\cS}_{\sigma}^I$
has a cuspidal simple submodule, then $I=\varnothing$. 
\end{itemize}
Now ($\dagger$) is proved as follows. Let $X\subseteq 
\overline{\cS}_{\sigma}^I$ be a cuspidal simple submodule. Using 
Lemma~\ref{sthc0}(i) and adjointness, we obtain that 
\[\Hom_{kG} (R_I^S(X),\overline{\cS}_{\sigma})\cong \Hom_{kL_I}(X,
\overline{\cS}_{\sigma}^I)\neq \{0\}.\] 
Thus, $\overline{\cS}_\sigma$ has a composition factor which is a
quotient of $R_I^S(X)$. Since $\overline{\cS}_\sigma$ and $\St_k
\subseteq kG\fb$ have the same composition factors, it follows that 
$kG\fb$ has a composition factor which is a quotient of $R_I^S(X)$. 
By our assumption and the characterisation of Harish-Chandra series in 
\cite[Theorem~5.8]{hiss2}, the pair $(I,X)$ is $N$-conjugate to a pair 
$(J,D_{\sigma}^J)$ where $J\in\cP_\sigma^*$. So there exists some $n\in N$ 
such that $nL_In^{-1}=L_J$ and ${^n\!X}\cong D_{\sigma}^J$. Using 
Lemma~\ref{sthc0}(ii), we conclude that $X\cong 
D_{\sigma}^I$. Thus, $D_{\sigma}^I$ is both isomorphic to a 
submodule and to a quotient of $\overline{\cS}_{\sigma}^I$. Now, having 
a unique simple quotient, the module $\overline{\cS}_{\sigma}^I$ is 
indecomposable. Hence, the multiplicity~$1$ statement in 
Proposition~\ref{twohiss} implies that $D_{\sigma}^I\cong
\overline{\cS}_{\sigma}^I$ and, consequently, we also have $D_{\sigma}^I
\cong \St_k^I\subseteq kL_I\fb_I$. Thus, $kL_I\fb_I\cong R_\varnothing^I
(k_H)$ has a cuspidal simple submodule. Again, by \cite[Theorem~5.8]{hiss2}, 
this can only happen if $I=\varnothing$.

(iii) By our assumption, the only composition factor of $\St_k$ which can 
possibly be cuspidal is $D_\sigma$. By (ii) and Proposition~\ref{twohiss}, 
$D_\sigma$ is not a composition factor of $\rad(\St_k)$. \qed
\end{proof}

\begin{rem}\label{sthc3a} In analogy to Example~\ref{socgln}, we define
$\cU_k(G)$ to be the set of all $Y\in \Irr_k(G)$ which are composition
factors of $kG\fb$. We have $\Irr_k(G\mid B)\subseteq \cU_k(G)$ but note
that, in general, we neither have equality nor is $\cU_k(G)$ the complete 
set of all {\em unipotent $kG$-modules} as defined, for example, in 
\cite[\S 1]{ghm1}. (Over $K$, we do have at least $\Irr_K(G\mid B)=
\cU_K(G)$.) For $J\subseteq S$, we define $\cU_k(L_J)$ analogously; the 
standard Borel subgroup of $L_J$ is given by $B_J:=B\cap L_J$. Let $X\in 
\cU_k(L_J)$ and $Y\in\cU_k(G)$. Then we have:

(a) All composition factors of $R_J^S(X)$ belong to $\cU_k(G)$.

(b) If $\Rst_J^S(Y)\neq \{0\}$, then all composition factors of 
$\Rst_J^S(Y)$ belong to $\cU_k(L_J)$.
\end{rem}

\begin{proof} (a) By the definitions, we have $kG\fb\cong 
R_{\varnothing}^S(k_H)$ and, similarly, $kL_J\fb_J\cong R_{\varnothing}^J
(k_H)$, where $\fb_J= \sum_{b\in B\cap L_J} b\in kL_J$. Hence, using 
transitivity, we obtain $kG\fb\cong R_J^S(kL_J\fb_J)$. Since Harish-Chandra 
induction is exact (see \cite[3.4]{my98}), $R_J^S(X)$ is 
a subquotient of $kG\fb$. 

(b) Since $kG\fb\cong R_\varnothing^S(k_H)$, the Mackey formula immediately
shows that $\Rst_J^S(kG\fb)$ is a direct sum of a certain number of copies
of $kL_J\fb_J$. It remains to use the fact that Harish-Chandra restriction 
is also exact. \qed
\end{proof}

\begin{exmp} \label{stgln} Let $G=\GL_n(q)$, where $n\geq 1$ and $q$ is any
prime power. Let $e\geq 2$ be defined as in Example~\ref{socgln}; also 
recall that $|\cU_k(G)|=\pi(n)$, where $\pi(n)$ denotes the number of 
partitions of $n$. By \cite[7.6]{ghm1}, $D_\sigma$ is cuspidal if and 
only if $n=1$ or $n=e\ell^j$ for some $j\geq 0$. (Note that, if $\ell
\mid q-1$, then our $e$ equals $\ell$, while it equals $1$ in 
[{\em loc.\ cit.}]; otherwise, the two definitions coincide.) Now, the 
$W$-conjugacy classes of subsets $J\subseteq S$ are parametrised by the 
partitions $\lambda \vdash n$ (see \cite[2.3.8]{gepf}); the Levi subgroup
$L_J$ corresponding to $\lambda$ is a direct product of general linear 
groups corresponding to the parts of $\lambda$. Hence, the subsets $J\in
\cP_\sigma^*$ are parametrized by the partitions $\lambda\vdash n$ such that 
each part of $\lambda$ is equal to $1$ or to $e\ell^j$ for some $j\geq 0$. 
So Remark~\ref{sthc3a}(a) and the counting argument in \cite[(2.5)]{ghm2} 
yield $\pi(n)$ simple modules in $\cU_k(G)$ which belong to $\Irr_k(G\mid 
J, \sigma)$ for some $J\in \cP_\sigma^*$. Thus, the hypothesis of 
Proposition~\ref{sthc3} is satisfied in this case. Consequently, $\St_k/
\rad(\St_k)$ is simple and $\St_k$ is multiplicity-free. (This was also
shown by Szechtman \cite{Sz1}, using different techniques.) Furthermore, 
the composition length of $\St_k$ is the coefficient of $t^n$ in the power
series
\[ \frac{1}{1-t}\prod_{j\geq 0} \frac{1}{1-t^{e\ell^j}}.\]
Indeed, let $c_n$ denote the composition length of $\St_k$. By
Proposition~\ref{sthc3}(i), $c_n$ equals the number of $J\in\cP_\sigma^*$ 
(up to $W$-conjugacy). By the above disussion (see also \cite[(2.5)]{ghm2}), 
this is equal to the number of sequences $(m_{-1},m_0,m_1,\ldots)$ of 
non-negative integers such that $m_{-1}+ em_0+e\ell m_1+\cdots =n$ 
(where $\GL_0(q)=\{1\}$ by convention). We multiply both sides by $t^n$ 
and sum over all $n\geq 0$. This yields 
\[\sum_{n\geq 0} c_nt^n =\sum_{(m_{-1},m_0,m_1,\ldots)} t^{m_{-1}+e
m_0+e\ell m_1+\ldots}=\Bigl(\sum_{m_{-1}\geq 0} t^{m_{-1}}\Bigr)
\Bigl(\sum_{m_0\geq 0} t^{em_0} \Bigr)\Bigl(\sum_{m_1\geq 0} 
t^{e\ell m_1} \Bigr)\cdots.\]
Using the identity $1/(1-t^r)=\sum_{i \geq 0} t^{ri}$ for all $r\geq 1$, 
we obtain the desired formula.
\end{exmp}

\begin{rem} \label{remsz} In the setting of Szechtman \cite{Sz1}, the 
above expression for $c_n$ means that the formula (4) in 
\cite[p.~605]{Sz1} holds for all~$n$. (Previously, this was only known 
for $n\leq 10$; see the remarks in [{\em loc.\ cit.}].) This 
formula gives an explicit expression of the layers in the Jantzen 
filtration of $\St_k$ (as defined by Gow \cite{gow}), as direct sums 
of simple modules. It also shows that the layers in this filtration are 
not always irreducible and, hence, Gow's conjecture \cite[6.3]{gow} does 
not hold in general. See also the explicit examples in \cite[\S 9]{Sz1}.
\end{rem}

\begin{exmp} \label{stclass} Let $G=G_n(q)$, $n\geq 1$, be one of the 
following finite classical groups:
\begin{itemize}
\item[(1)] The general unitary group $\GU_n(q)$ for any $n$, any $q$.
\item[(2)] The special orthogonal group $\mbox{SO}_{n}(q)$ where $n=2m+1$ 
is odd and $q$ is odd. 
\item[(3a)] The symplectic group $\mbox{Sp}_{n}(q)$ where $n=2m$ is 
even and $q$ is a power of $2$. 
\item[(3b)] The conformal symplectic group $\mbox{CSp}_{n}(q)$ where
$n=2m$ is even and $q$ is odd. 
\item[(4)] The conformal orthogonal group $\mbox{CSO}_n^{\pm}(q)$ where
$n=2m$ is even and $q$ is odd. 
\end{itemize}
Each of these groups can be realized as $G=\bG^F$ where $\bG$ has a 
connected center and $\bG$ is simple modulo its center. By convention, 
$G_0(q)$ is the trivial group, except for the conformal groups, where it is 
cyclic of order $q-1$. We define the parameter $\delta$ to be $2$ in case 
(1) and to be $1$ in all the remaining cases. 
\end{exmp}

\begin{thm}[\protect{\cite{lymgh}}, Gruber \protect{\cite{grub}}, and 
Gruber--Hiss \protect{\cite{grhi}}] \label{grhi1} Let $G=G_n(q)$ be as in 
Example~\ref{stclass} and assume that $\ell$ is ``linear'', that is, 
$q^{\delta m-1} \not\equiv -1\bmod \ell$ for all $m\geq 1$.  
\begin{itemize}
\item[{\rm (i)}] We have $|\Irr_\C(W)|=|\cU_k(G)|$.
\item[{\rm (ii)}] If $Y\in \cU_k(G)$, then $\Rst_J^S(Y)\neq \{0\}$ for
some subset $J\subseteq S$ such that $L_J$ is a direct product of groups 
of untwisted type $A$.
\end{itemize}
\end{thm}

\begin{proof} This follows from \cite[\S 4]{lymgh} in the 
cases (1), (2), (3a), (3b), and from \cite{grub} in case (4). We shall
refer to the more general setting in \cite{grhi} (where all of $\Irr_k(G)$ is
considered). 

(i) Note that, by the ``classical fact'' in characteristic
$0$ mentioned in \ref{princ}, we certainly know that $|\Irr_\C(W)|=
|\cU_K(G)|$. Hence, the assertion immediately follows from the block 
diagonal shape of the decomposition matrix in \cite[Theorem~8.2]{grhi}.

(ii) Let $Q$ be a projective indecomposable $\cO G$-lattice such that 
$\overline{Q}$ is a projective cover of~$Y$. By \cite[Cor.~8.6]{grhi}, 
$Q$ is a direct summand of $R_J^S(Q')$, where $J\subseteq S$ and $Q'$ is 
a projective indecomposable $\cO L_J$-lattice such that the following
conditions hold. First, we have $L_J\cong G_a(q) \times L_\lambda$ where 
$n=a+2m$ ($a,m\geq 0$) and $G_a(q)$ is a group of the same type as $G$; 
furthermore, $\lambda$ is a composition of~$m$ and $L_\lambda$ is a 
direct product of general linear groups $\GL_{\lambda_i}(q^\delta)$ 
where $\lambda_i$ runs over the non-zero parts of $\lambda$. Finally, 
under the isomorphism $L_J\cong G_a(q) \times L_\lambda$, we have 
$Q'\cong Q_a' \otimes Q_\lambda'$ where $Q_a'$ is a projective 
indecomposable $\cO G_a(q)$-lattice such that $KQ_a'$ has only cuspidal 
constituents and $Q_\lambda'$ is an indecomposable direct summand of 
the Gelfand-Graev lattice for $\cO L_\lambda$.  

Now, since $\overline{Q}$ is a direct summand of $R_J^S(\overline{Q}')$,
we have $\Hom_{kL_J}(\overline{Q}',\Rst_J^S(Y))\neq \{0\}$ by adjointness. 
This shows, first of all, that $\Rst_J^S(Y) \neq \{0\}$. Using  
Remark~\ref{sthc3a}(b), we conclude that $\Hom_{kL_J}(\overline{Q}',
kL_J\fb_J)\neq 0$. Consequently, since $Q'$ is projective, we also have 
$\Hom_{KL_J}(KQ',KL_J\fb_J)\neq 0$. So, by the above direct product 
decomposition of $L_J$, at least one of the cuspidal composition factors 
of $KQ_a'$ belongs to $\cU_K(G_a(q))$. But this can only happen if 
$G_a(q)$ has $BN$-rank equal to~$0$. Thus, $L_J$ has the required form. \qed
\end{proof}

\begin{abs} \label{grhi2} Let $G=G_n(q)$ be as in Example~\ref{stclass}
and assume that $\ell$ is linear. By Theorem~\ref{grhi1}(ii) and the
characterisation of Harish-Chandra series in \cite[Theorem~5.8]{hiss2}, 
every $Y\in \cU_k(G)$ is a submodule of $R_J^S(X)$ where $J\subseteq S$ 
is such that $L_J$ is isomorphic to a direct product of groups of untwisted
type $A$, and $X\in \Irr_k(L_J)$ is cuspidal. Then, by adjointness, $X$ 
is a composition factor of $\Rst_J^S(Y)$ and, hence, $X\in \cU_k(L_J)$ by 
Remark~\ref{sthc3a}(b). But then the known results on 
general linear groups imply that $X\cong D_\sigma^J$; see, e.g., 
\cite[Cor.~6.16]{my98}. Thus, the hypothesis of Proposition~\ref{sthc3} is 
satisfied. (This is also mentioned in Dipper--Gruber \cite[4.22]{digr},
with only a sketch proof.)

Thus, in all the cases listed in Example~\ref{stclass}, $\St_k$ is 
multiplicity-free, $\St_k/\rad(\St_k)$ is simple and the composition 
length of $\St_k$ is determined as in Proposition~\ref{sthc3}(i). 
Consequently, one can derive a generating function for the composition 
length of $\St_k$, similar to that for $\GL_n(q)$ in Example~\ref{stgln}. 
We will only give the details for $G=\GU_n(q)$.  Write $n=2m$ (if $n$ is 
even) or $n=2m+1$ (if $n$ is odd); furthermore, since $\delta=2$, we set 
\[ \tilde{e}:=\min\{i\geq 2\mid 1+q^2+q^4+\cdots +q^{2(e-1)}\equiv 0
\bmod \ell\}\]
in this case. We can now use the counting argument in the proof of 
\cite[Theorem~4.11]{ghm2}; see also \cite[\S 4]{lymgh}. This shows 
that the subsets $J\in \cP_\sigma^*$ are parametrized (up to 
$W$-conjugacy) by the partitions $\lambda\vdash m$ such that each part 
of $\lambda$ is equal to $1$ or to $\tilde{e}\ell^j$ for some $j\geq 0$. 
So the number of $J\in\cP_\sigma^*$ (up to $W$-conjugacy) is equal to 
the number of sequences $(m_{-1},m_0,m_1, \ldots)$ of non-negative 
integers such that $m_{-1}+\tilde{e}m_0+ \tilde{e}\ell m_1+\cdots =m$. 
Thus, we find that the composition length of $\St_k$ for $G=\GU_n(q)$ is 
the coefficient of $t^m$ (and not $t^n$ as in Example~\ref{stgln}) in the 
power series
\[ \frac{1}{1-t}\prod_{j\geq 0} \frac{1}{1-t^{\tilde{e}\ell^j}}\qquad 
\mbox{(assuming that $\ell$ is linear for $G$)}.\]
\end{abs}

\begin{rem} \label{remdigr} Within the much more general setting of 
Dipper--Gruber \cite{digr}, we have considered here the ``projective 
restriction system'' $\cPR(X_G,Y_L)$ where
\[X_G:=\cS_\sigma, \quad L=H\quad \mbox{and}\quad Y_L=\cO H
\;\;\mbox{(regular $\cO H$-module)}.\]
In this particular case, the arguments in [{\em loc.\ cit.}] drastically 
simplify, and this is what we have tried to present in this section. We 
note, however, that these methods only yield quite limited information 
about $\St_k$ when $\ell$ is not a ``linear prime''. Only two of the 
composition factors in the socle series displayed in 
Example~\ref{soctrivrk1} are accounted for by these methods (namely, 
$k_G$, $\varphi_3$ for ${^2G}_2(q^2)$ and $k_G,\vartheta$ for 
$\GU_3(q)$); all the remaining composition factors are cuspidal. Also note
that, in these examples, $\St_k$ is not multiplicity-free.
\end{rem}

\smallskip
\noindent {\bf Acknowledgements.} I wish to thank Gerhard Hiss for
clarifying discussions about the contents of \cite{digr}, \cite{grub}, 
\cite{grhi}.

%%%%%%%%%%%%%%%%%%%%%%%%%%%%%%%%%%%%%%%%%%%%%%%%%%%%%%%%%%%%%%%%%%%%%%%

\end{document}